\documentclass[12pt]{amsart}
\usepackage{palatino}
\usepackage{amscd,amssymb}
\usepackage{float}
\usepackage{graphicx}
\usepackage{dbnsymb}

\usepackage[cmtip,matrix,arrow]{xy}

\usepackage{graphicx}

\numberwithin{equation}{section}

\newtheorem {Theorem}                   {Theorem}
\newtheorem {varTheorem}                {Theorem}
\newtheorem {RefTheorem}[equation]      {Theorem}

\theoremstyle{definition}
\newtheorem {Definition}[equation]{Definition}

\newtheorem {Question}[equation]      {Question}
\newtheorem{zremark}[equation] {Remark}
\newenvironment{Remark} {\par\footnotesize\zremark}{~\\}{\endzremark}

\newcommand{\pr} {\smallskip\noindent{\bf Proof\,\,}}

\newcommand     {\comment}[1]   {}
\newcommand     {\mute}[2] {}
\newcommand     {\printname}[1] {}

\newcommand{\labell}[1] {\label{#1}\printname{#1}}

\def    \hol    {{\operatorname{hol}}}

\def    \to     {\longrightarrow}

\def    \C      {{\mathbb C}}
\def    \R      {{\mathbb R}}

\def    \Z      {{\mathbb Z}}

\def    \cA     {{\mathcal A}}
\def    \bA     {{\bf{ A}}}
\def    \cG    {{\mathcal G}}

\def    \fg     {{\mathfrak g}}

\def    \tr     {{\rm {tr}~}}
\def    \hol   {{\rm {hol}}}
\def    \pr     {\operatorname{pr}}

\begin{document}

\title[Deformation quantization and Chern Simons Gauge Theory] {On deformation quantization of the space of connections on a two manifold and Chern Simons Gauge Theory}

\author[Jonathan Weitsman]{Jonathan Weitsman}
\thanks{Supported in part by a grant from the Simons Foundation (\# 579801)}
\address{Department of Mathematics, Northeastern University, Boston, MA 02115}
\email{j.weitsman@neu.edu}
\thanks{\today}

\begin{abstract}  

We use recent progress on Chern-Simons gauge theory in three dimensions \cite{kqm} to give explicit, closed form formulas for the star product on some functions on the affine space $\cA(\Sigma)$ of (smooth) connections on the trivialized principal $G$-bundle on a compact, oriented two manifold $\Sigma.$   These formulas give a close relation between knot invariants, such as the Kauffman bracket polynomial, and the Jones and HOMFLY polynomials, arising in Chern Simons gauge theory, and deformation quantization of $\cA(\Sigma).$  This relation echoes the relation between the manifold invariants of Witten \cite{Witten} and Reshetikhin-Turaev \cite{rt} and {\em geometric} quantization of this space (or its symplectic quotient by the action of the gauge group).  In our case this relation arises from explicit algebraic formulas arising from the (mathematically well-defined) functional integrals of \cite{kqm}.
\end{abstract}
 
\maketitle

\tableofcontents

\section{Introduction}

In \cite{bffls} Bayen, Flato, Fronsdal, Lichnerowicz, and Sternheimer proposed deformation quantization as a method of constructing algebraically the effect of quantization by deforming the algebra $C^\infty(M)$ of functions on a symplectic manifold $(M,\omega),$ or Poisson manifold $(M,\Pi).$  They conjectured, and in some cases proved, that on such a manifold there exists an associative product $\star_h$ on the formal power series $C^\infty(M)[[h]],$ with the property that 

$$ f \star_h g = fg + h \{f,g\} + O(h^2)$$

\noindent for $f,g \in C^\infty(M)[[h]],$  where $\{\cdot,\cdot\}$ is the Poisson bracket on $M.$

The construction of the star product for symplectic manifolds in general was given by deWilde and Lecomte \cite{dl} and by Omori, Maeda, and Yoshioka \cite{omy} and Fedosov \cite{f}, and for general Poisson manifolds by Kontsevich \cite{k}.  The resulting star products are often not so easy to compute.

The purpose of this paper is to use recent progress on Chern-Simons gauge theory in three dimensions \cite{kqm} to give explicit, closed form formulas for the star product on some functions on the affine space $\cA(\Sigma)$ of (smooth) connections on the trivialized principal $G$-bundle on a compact, oriented two manifold $\Sigma.$   These formulas give a close relation between knot invariants arising in Chern Simons gauge theory and deformation quantization of $\cA(\Sigma),$ echoing the relation between the manifold invariants of Witten \cite{Witten} and Reshetikhin-Turaev \cite{rt} and {\em geometric} quantization of this space (or its symplectic quotient by the action of the gauge group).  In our case this relation arises from explicit algebraic formulas in the quantum field theory of \cite{kqm}, which is mathematically well defined.

Let $\Sigma$ be a compact, oriented two manifold, and $G$ a compact Lie group.\footnote{We will later generalize this to the case of noncompact classical groups.} Choose a representation of $G$ so that the resulting trace ${\rm Tr}$ gives a symplectic form $\Omega$ on the space $\cA(\Sigma)$ of connections on the trivial principal $G$-bundle on $\Sigma.$  In more detail, if we use the trivialization to identify $\cA(\Sigma) = \Omega^1(\Sigma) \otimes \fg,$ the symplectic form is given in \cite{ab} by 

$$ \Omega|_A (v,w) = {\rm Tr} \int_\Sigma v \wedge w$$

\noindent where $A \in \cA(\Sigma)$ and $v,w \in T\cA(\Sigma)|_A = \Omega^1(\Sigma)\otimes\fg.$  The explicit formula shows that the symplectic form $\Omega$ is invariant under affine translations on $\cA(\Sigma).$\footnote{The symplectic form is also invariant under the action of the gauge group.  This action does not commute with the affine translations, and this fact will have consequences we will study later in the paper.  See Section \ref{R2sec}.}

The Poisson bivector field on $\cA(\Sigma)$ can be written down in local coordinates; we have

\begin{equation}\labell{pb0}\Pi =   \int_\Sigma dx \sum_{\alpha,\beta} \sum_{a,b} \epsilon_{ab}   (g^{-1})_{\alpha\beta} \frac{\delta}{\delta A_\alpha^a(x)}\frac{\delta}{\delta A_\beta^b(x)}\end{equation}
 
 \noindent where we choose a basis $\{e_\alpha\}$ for $\fg,$ and write $g_{\alpha\beta}= {\rm Tr~} (e_\alpha e_\beta).$

Since $\cA(\Sigma)$ is an affine space, and since $\Omega$ is translation invariant, the star product can be computed using formulas discovered by \cite{bffls} in their original paper\footnote{See also \cite{f}, p. 222}.  Explicitly, denote by $\Pi$ the Poisson bivectorfield associated to the symplectic form $\Omega.$  Then

\begin{equation}\labell{affstar} f \star_h g = m e^{h \Pi } f \otimes g,\end{equation}

\noindent where $m$ denotes the multiplication map on functions.

The expression (\ref{affstar}) has the appearance of a canonical object, but in fact it involves an implicit choice.  The reason is that Fedosov's construction of the star product depends on a choice of a symplectic connection on the underlying manifold; that is, a torsion-free connection on the tangent bundle preserving the symplectic form.  In the case of an affine space, there is a natural choice of such a connection, and it leads to the formula (\ref{affstar}).  But there are other choices; in particular the choice of the affine connection on $\cA(\Sigma)$ is not a gauge invariant choice.  We will see that this has consequences for the behavior of the star product under reduction by the action of the gauge group.

The affine space $\cA(\Sigma)$ is equipped with some functions (the {\em Wilson loops}) arising from curves on $\Sigma.$  Choose a representation $R$ of $G$ (which may or may not be the same as the one used to define the symplectic form).  Given a closed, oriented curve $C$ in $\Sigma,$ define $W_C : \cA(\Sigma) \to \C$ by

$$W_C(A) = {\rm tr}_R {\rm hol}_{C,\bullet} (A)$$

\noindent where $\bullet$ is any basepoint on $C$ and ${\rm tr}$ is the trace in the representation $R.$  (To simplify the notation, we will suppress the basepoint and the representation when possible.)

The main result of this paper will be a computation of the star product on functions generated by the $W_C;$ we have to impose intersection conditions on the curves appearing in any given star product.  In view of (\ref{affstar}), the first task is to compute the Poisson bracket 

$$\{W_C, W_{C^\prime}\}$$

\noindent of two such functions, which from now on we assume to intersect transversally.  Consider the case where $G = GL(n,\C),$ the representations for both the symplectic form and the Wilson loops are given by the defining representation, and two curves $C, C^\prime$ intersect at one point $p.$  We will show that

\begin{equation}\labell{pb1}  \{ W_C, W_{C^\prime}\} = W_{C *_p C^\prime}\end{equation}

\noindent where ${C *_p C^\prime}$ is the curve obtained by concatenating the oriented curves $C$ and $C^\prime$ at their intersection point $p.$\footnote{The star notation is used in the literature to denote both the star product in deformation quantization, and the concatenation of loops.  We have maintained this usage.  I do apologize.}  More generally, we will show (see (\ref{pbgln})) that if there is more than one intersection point, the Poisson bracket is given by the sum of the Wilson loops obtained from concatenating the two curves at each of the intersection points.  We will obtain similar formulas for the Poisson bracket for the groups $SU(2)$ and $SL(2,\R),$ though these involve unoriented curves.

In all these cases our formulas for the Poisson bracket coincide with those discovered by Goldman \cite{goldman} for the Poisson bracket of the Wilson loops, considered as functions on the Marsden-Weinstein quotient $\cA(\Sigma)//\cG$ of the space of connections by the action of the gauge group $\cG= {\rm Map}(\Sigma,G).$  This space is the character variety ${\rm Hom}(\pi_1(\Sigma), G)/G$ of conjugacy classes of representations of the fundamental group of $\Sigma$ in $G.$  The reason for the agreement between our formulas and Goldman's is that the both the functions $W_C$ and the symplectic form, and hence the Poisson bracket, are gauge invariant; the identity between Poisson brackets of invariant functions on a Hamiltonian $G$-space and the corresponding Poisson brackets on the reduced space follows by general principles of symplectic geometry.  This is not the case for star products, unless additional conditions are imposed.

In addition to the resemblance to Goldman's results, there is also a close similarity between the Poisson bracket (\ref{pb1}) and formulas arising in Chern-Simons gauge theory \cite{kqm}, generalized from the case of $\R^3$ to that of $\Sigma \times \R.$  Consider the space of connections on the trivialized prinicipal $G$-bundle on $\Sigma \times \R.$  If we impose the axial gauge condition $A_3 = 0, $ the Chern-Simons action on the resulting space $\bA$ of compactly supported connections is given by

$${\rm CS}(A) = \frac12 {\rm Tr} \int_{\Sigma \times \R} A \wedge \frac{dA}{dt}\wedge dt.$$

The propagator for this theory is obtained by inverting the differential operator in the quadratic action ${\rm CS}(A)$ and is given by \cite{kqm}

\begin{equation}\labell{propagator} 
\Gamma^{ij}_{\alpha\beta} (x,t; x^\prime,t^\prime) = (g^{-1})_{\alpha\beta}\epsilon_{ij} \delta(x,x^\prime) u(t-t^\prime).
\end{equation}

\noindent Here $x,x^\prime\in \Sigma,$  $t,t^\prime \in \R,$ and we have chosen a basis $\{e_\alpha\}$ for $\fg,$ and define $g_{\alpha \beta} = {\rm Tr~}(e_\alpha e_\beta);$\footnote{The trace is taken in the representation used to define the Chern Simons function.}  likewise we choose an oriented basis $f_1,f_2$ for the fibres of the cotangent bundle $T^*\Sigma|_x  $ at a point $x$ and its pullback to $\Sigma \times \R,$ agreeing with the orientation of $\Sigma.$  Finally\footnote{The addition of a constant to $u,$ although in principle allowable, results in no change to any expectation, due to the fact that $\epsilon_{ij}$ is skew.} 

$$u(s) =  \left\{
	\begin{array}{ll}
		\frac12  & \mbox{if } s \geq 0 \\
		-\frac12 & \mbox{if } s< 0
	\end{array}
\right.$$

We note that there is a resemblance between the propagator (\ref{propagator}) and the differential operator appearing in the Poisson bivector field (\ref{pb1}).  To exploit the similarity, consider the operator

\begin{equation}\labell{op2} D=  \int dx dx^\prime dt dt^\prime \sum_{\alpha,\beta} \sum_{a,b} \Gamma^{ab}_{\alpha\beta} (x,t; x^\prime,t^\prime)  \frac{\delta}{\delta A_\alpha^a(x,t)}\frac{\delta}{\delta A_\beta^b(x^\prime,t^\prime)},\end{equation}

\noindent or, more explicitly

\begin{equation}\labell{op} D=  \int dx dx^\prime dt dt^\prime \sum_{\alpha,\beta} \sum_{a,b} (g^{-1})_{\alpha\beta} \epsilon_{ab}  u(t  - t^\prime) \delta(x,x^\prime) \frac{\delta}{\delta A_\alpha^a(x,t)}\frac{\delta}{\delta A_\beta^b(x^\prime,t^\prime)}. \end{equation}

\noindent The operator $D$ is the operator which evaluates expectations in Chern-Simons gauge theory.  And we have, and for any two curves $C, C^\prime$ in $\Sigma,$ for any $\epsilon > 0,$ and any $A \in \cA,$

\begin{varTheorem}For any two curves $C, C^\prime$ in $\Sigma$ intersecting transversally, for any $\epsilon > 0,$ and any $A \in \cA(\Sigma),$
\begin{equation}\labell{anid} 
m \Pi (W_C \otimes W_{C^\prime})  (A) = \frac12 D (W_{T_\epsilon C}  W_{T_{-\epsilon} C^\prime}) (\pi^* A)
\end{equation}with similar formulas for the Poisson brackets of functions given by polynomials in Wilson loops.
\end{varTheorem}
\noindent Here for any curve $C$ in $\Sigma,$ the curve $T_\epsilon C$ denotes the curve in $\Sigma \times \R$ obtained by translating $C$ by $\epsilon$ in the vertical direction\footnote{We consider self crossings of the individual curves $C$ and $C^\prime$ as virtual crossings, to which the operators $D$ are not to be applied. }, and  $\pi : \Sigma \times \R \to \Sigma$ is the projection.  Likewise

\begin{varTheorem}For any two curves $C, C^\prime$ in $\Sigma$ intersecting transversally, for any $\epsilon > 0,$ and any $A \in \cA(\Sigma),$
\begin{equation}\labell{anidstar} 
m e^{h \Pi} (W_C \otimes W_{C^\prime} ) (A) =   e^{\frac{h}{2} D} (W_{T_\epsilon C}  W_{T_{-\epsilon} C^\prime})(\pi^* A),
\end{equation}

\noindent with similar formulas for the Poisson brackets of functions given by polynomials in Wilson loops. \end{varTheorem}

Thus we may define the star product on $\cA(\Sigma)$ using expectations in Chern-Simons gauge theory, that is 

\begin{Definition}For any two curves $C, C^\prime$ in $\Sigma,$ for any $\epsilon > 0,$ and any $A \in \cA,$ define

$$W_C \star_h W_{C^\prime} =  e^{\frac{h}{2} D} (W_{T_\epsilon C}  W_{T_{-\epsilon} C^\prime})(\pi^* A)$$
\end{Definition}

\noindent In Section \ref{combisec} we will prove\footnote{Since $\cA(\Sigma)$ is an infinite dimensional space, and since our star product is not defined everywhere, this does not quite follow from the results of \cite{bffls,f}.}

\begin{varTheorem} The star product is associative where defined; that is, if $C,C^\prime,C^"$  are curves in $\Sigma$  which intersect pairwise transversally, then

$$(W_C \star_h W_{C^\prime} ) \star_h W_{C^"}  = W_C \star_h( W_{C^\prime}\star_h W_{C^"}).$$\end{varTheorem}

This relation between deformation quantization of the space of connections on a two manifold $\Sigma$ and Chern-Simons gauge theory on $\Sigma \times \R$ is the main result of this paper.  Now in \cite{kqm} we found explicit formulas for evaluating the expectations of Wilson loops in Chern-Simons gauge theory in axial gauge, such as the expressions

$$e^{\frac{h}2 D} (W_{T_\epsilon C}  W_{T_{-\epsilon} C^\prime}).$$

The resulting formulas turned out to be close relatives of the Jones and HOMFLY polynomials.  We therefore obtain as a consequence of this relation explicit formulas, related to the Jones and HOMFLY polynomials, for the star product of Wilson loops on the space of connections (See Section \ref{ex1}).  These arise from the formulas derived in \cite{kqm} for expectations in Chern Simons gauge theory; explicitly, these are formulas 
(\ref{skeinsu2over2}) and (\ref{skeinsu2under2}) (in the case of $SU(2), SL(2,\R),$ and $SL(2,\C)$) or formulas (\ref{oskeingln}) and (\ref{uskeingln}) (in the case of $GL(n,\C)$ or $U(n)$) 

In view of the fact that the Poisson bracket of Wilson loops on the space of connections agreed with Goldman's formula for the Poisson bracket of those functions, considered as functions on the moduli space of flat connections, it is natural to ask whether a similar phenomenon occurs for the star product, and whether our formulas give a deformation quantization of these moduli spaces; in the case of $G=SL(2,\R),$ we would then obtain a deformation quantization of Teichmuller space.  And in \cite{f2} Fedosov shows that, in the case of finite dimensional symplectic manifolds, deformation quantization commutes with reduction, and that in an appropriate sense, the star product on a Hamiltonian $G$-space, restricted to invariant functions, gives a star product on the symplectic quotient.

However a closer look at the formulas formulas 
(\ref{skeinsu2over2}) and (\ref{skeinsu2under2}) (in the case of $SU(2), SL(2,\R),$ and $SL(2,\C)$) and (\ref{oskeingln}) and (\ref{uskeingln}) (in the case of $GL(n,\C)$ or $U(n)$)  shows that the star product of two Wilson loops $W_C, W_{C^\prime}$ is not invariant under homotopy.  This causes no difficulty for the star product on $\cA(\Sigma)$, but on the symplectic quotient, homotopic loops give rise to the same function on the moduli space of flat connections, so that the star product must be homotopy invariant to be well defined.

The failure of homotopy invariance was already noted in \cite{kqm} as the failure of the expectations in Chern-Simons gauge theory to be invariant under the second Reidemeister move.  In \cite{kqm} we showed that, for $G=SU(2)$ this could be remedied by using a background gauge field to shift the value of the expectation of a single unknotted circle from its expected value in $\langle \bigcirc \rangle = 2$ in gauge theory to the value $\langle\bigcirc\rangle = -(q^2 + q^{-2})$ 
 expected\footnote{Here $q$ is a parameter appearing both in Kauffman's bracket polynomial and in Chern-Simons gauge theory.  See \cite{K2} and \cite{kqm}.} from Kauffman's bracket polynomial \cite{K2}.\footnote{In \cite{kqm} there are normalization conventions slightly different from Kauffman's which adjust these expectations by a sign.}
From the point of view of Fedosov's construction, the fact that the star product on $\cA(\Sigma)$ does not descend in a direct way to the quotient may be expected due to the fact that our construction of the star product used implicitly the affine symplectic connection on $\cA(\Sigma),$ which is not gauge invariant.  We would hope that this point of view may be helpful also for Chern-Simons gauge theory.  In particular it would be interesting to see whether a more sophisticated handling of the time independent gauge transformations in Chern-Simons gauge theory in axial gauge may give insight into the rather mysterious shift  $\langle\bigcirc\rangle = -(q^2 + q^{-2}).$

The structure of this paper is as follows.  In Section \ref{goldmansec} we repeat the computation from \cite{kqm} of the Poisson bracket of functions given by polynomials in Wilson loops, giving formulas that agree with those computed by Goldman on the moduli space of flat connections.  In Section \ref{kqmsec} then recall the results of \cite{kqm}, and extend them to the manifold $\Sigma \times \R;$ this extension is straightforward.  In Section \ref{combisec} we compare Chern-Simons functional integral on $\Sigma \times \R$ with deformation quantization of  $\cA(\Sigma),$ and show that the star product of Wilson loops arising from transversally intersecting curves, considered as functions on $\cA(\Sigma),$ may be computed using the expectations defined and computed in \cite{kqm}.  We state this as Definition \ref{stp}.  We then prove that associativity of the star product, where defined, follows from properties of expectations in the linear functional corresponding to Chern-Simons functional integral.  We close the section with an algorithm which gives an explicit formula for the star product of Wilson loops.  Finally in Section \ref{gq}, we briefly review the deformation quantization of Kahler manifolds arising from geometric quantization in the works of Boutet de Monvel and Guillemin \cite{bg}, Bordemann, Meinrenken and Schliechenmeier \cite{bms} and Guillemin \cite{guill}. We show how applying these results to the moduli space of flat connections on a Riemann surface hints at a relation
between deformation quantization of the moduli space of flat connections on a Riemann surface and the topological quantum field theory arising from geometric quantization of that space.  This parallels the relation we found in this paper
between deformation quantization of the space of connections and Chern-Simons gauge theory.   We next discuss the relation between the (non)-invariance of the expectations in the Chern-Simons gauge theory of \cite{kqm} under the Second Reidemeister move and the (non)-invariance of the star product on $\cA(\Sigma)$ under reduction.   While the lack of invariance of expectations in Chern-Simons gauge theory is mysterious, the non-invariance of the star product arises from a known cause---the gauge non-invariance of the affine symplectic connection on $\cA(\Sigma).$  So this hints at a possible resolution of both issues, as mentioned above.  Finally we remark briefly on the known relation between skein modules and quantized Teichmuller space.  From our point of view this provides evidence that the issues discussed above are solvable.  Our results would then provide a geometric proof, based on mathematical functional integrals, of the relation between Chern Simons gauge theory and deformation quantization of moduli spaces (such as Teichmuller space).  We also ask whether the coincidence we see between the results for $G = SU(2)$ (related to quantum invariants or the Kauffman bracket) and those for $G = SL(2,\R)$ (related to Teichmuller space) would give some insight into the volume conjecture.  We close with an appendix where we construct various coproducts on the functions we have studied, and ask whether any of those are related to Turaev's coproduct.
 
\section{The Poisson Bracket}\labell{goldmansec}

As a first step toward computing the star product on $\cA(\Sigma),$ we compute the Poisson bracket of two Wilson loops on this space.  We perform these computations by methods similar to those used in \cite{kqm} for expectations in Chern-Simons gauge theory; in Section \ref{combisec} we will see that this apparent analogy is in fact an identity.  We will see that these computations give results identical to those given by Goldman \cite{goldman} on the moduli space of flat connections on the two manifold.  Since this moduli space is the reduced space of the space of connections by the action of the gauge group, and since the Wilson loops are invariant under this group, this is to be expected.

Let $\Sigma$ be a closed, oriented two manifold, let $G$ be a Lie group as before, along with a choice of representation whose associated trace gives a nondegenerate inner product on $\fg,$ and let $\cA(\Sigma)$ be the space of connections on the trivialized principal $G$-bundle on $\Sigma.$  Then $\cA(\Sigma) = \Omega^{1}(\Sigma)\otimes \fg.$  This space is equipped \cite{ab} with a symplectic form $\Omega$ given by

$$ \Omega_A (v,w) =  {\rm Tr} \int_\Sigma v \wedge w$$

\noindent  where $A \in \cA(\Sigma)$ and $v,w \in T\cA_A = \Omega^{1}(\Sigma)\otimes \fg,$ and the trace ${\rm Tr}$ is the trace in the chosen representation.

The Poisson bivector field $\Pi$ arising from the symplectic form $\Omega$ is given in local coordinates  by 

\begin{equation}\labell{poiss}\Pi =   \int_\Sigma dx \sum_{\alpha,\beta} \sum_{a,b} \epsilon_{ab}   (g^{-1})_{\alpha\beta} \frac{\delta}{\delta A_\alpha^a(x)}\frac{\delta}{\delta A_\beta^b(x)}\end{equation}
 
\noindent Here we have chosen a basis $\{e_\alpha\}$ for $\fg,$ and write $g_{\alpha\beta}= {\rm Tr~} (e_\alpha e_\beta).$

\begin{Remark}\labell{invtrmk} An invariant expression for $\Pi$ is given as follows.  Given a differentiable function $f$ on $\cA(\Sigma),$ the derivative is $\delta f |_A \in \Omega^1(\Sigma,\fg)^*.$  The orientation of $\Sigma$ gives a map $\Phi: \Omega^1(\Sigma,\fg)\to \Omega^1(\Sigma,\fg)^*$ given by $\Phi(v)(w) = {\rm Tr} \int_\Sigma v \wedge w$ for $v,w \in \Omega^1(\Sigma,\fg).$  In these terms, morally $\Pi (fg) (A) ={\rm Tr} \int_\Sigma \Phi^{-1}(\delta f|_A) \wedge  \Phi^{-1}(\delta g|_A),$ for any differentiable functions $f, g$ on $\cA(\Sigma)$ whose derivatives are in the image of $\Phi.$  (In general, these derivatives may correspond under $\Phi$ to distributional one forms.) This coordinate free definition agrees with (\ref{poiss}).\end{Remark}
  
Suppose that we are given two oriented simple closed curves $C,C^\prime$ in $\Sigma,$ intersecting transversally at a finite number of points $x_1,\dots,x_n.$   We wish to compute the Poisson bracket

$$\{ W_C, W_{C^\prime}\}$$

\noindent (where as before we write $W_C(A) = {\rm tr}_R \hol_{C,\bullet}(A)={\rm tr~} \hol_{C,\bullet}(A)$ for the trace in some representation $R$ of the holonomy of the connection $A$ along the oriented path given by $C$ with basepoint $\bullet,$ and suppress for notational simplicity the basepoint where we are taking a trace, as well as the representation) by applying the skew operator $\Pi$ to the pair of functions $W_C,W_{C^\prime}.$

We obtain

\begin{equation}\labell{pb}\{ W_C , W_{C^\prime} \} (A)=  \sum_i \sum_{\alpha,\beta}  \sum_{a,b} (g^{-1})_{\alpha\beta}  \epsilon_{ab}  \frac{\delta W_C(A)}{\delta A_\alpha^a(x_i)}\frac{\delta W_{C^\prime}(A)}{\delta A_\beta^b(x_i)}  .
\end{equation}

\noindent where the $x_i$ are the transverse intersection points of $C$ and $C^\prime.$  

To perform computations of this type, we recall some formulas from \cite{kqm}.

We begin by studying connections on the circle.  Let $A$ be a connection on the trivial principal $G-$bundle on $C=S^1,$ with its standard orientation, and let $v \in \Omega^1(S^1)\otimes \fg.$  We compute

\begin{equation}\labell{conn1}\frac{d}{ds}\vert_{s=0} {\rm tr~} \hol_C (A + s v) = \int_0^{2 \pi} dt~ {\rm tr~}( \hol_{C,t} (A) v(t) ).\end{equation}

Here we have written $\hol_{C,t} (A) \in G $ for the holonomy of the connection $A$ along the circle from the point $t$ to itself in the direction given by the standard orientation, and we consider both $G$ and $\fg$ as subsets of some space $M_n(\C)$ of $n \times n$ matrices corresponding to the chosen representation.

In terms of functional derivatives, this equation (\ref{conn1}) can be written

\begin{equation}\labell{conn}
\frac{\delta}{\delta A_\alpha(t)} {\rm tr~} \hol_C(A) = {\rm tr~} (p( \hol_{C,t} (A)) e_\alpha)\end{equation}

\noindent where $p: G \to \fg$ is the projection arising from the representation we chose for $G$ (inside some $GL(n,\C)$) and the inner product arising from the corresponding trace.

We also record this functional derivative for the case of a connection on $[0,1].$  We have (using again the representation in $GL(n,\C)$)\footnote{The factor of $\frac12$ in the equations below is due to the derivative being taken at the endpoint, and can be seen by applying the analog of (\ref{conn1}) to a series of approximations of the delta function by smooth functions.  If the connection on a circle is thought of as being a connection on the interval with endpoints identified, the factors of $\frac12$ arising from the endpoints add to give (\ref{conn1}).}

\begin{equation}\labell{connpath}
\frac{\delta}{\delta A_\alpha(0)}   \hol_{[0,1]} (A) = \frac12 e_\alpha  \hol_{[0,1]} (A)   \in  M_n(\C). \end{equation}

\noindent  Likewise
 
\begin{equation}\labell{connpathe}
\frac{\delta}{\delta A_\alpha(1)}   \hol_{[0,1]} (A) =  \frac12 \hol_{[0,1]} (A) e_\alpha  \in M_n(\C) \end{equation}

Given an embedded path, or curve, in a manifold, and a connection on the trivialized principal $G$ bundle on the manifold, these formulas give expressions for the derivatives in the tangent direction to the curve of the holonomy of the pulled back connection.

We now return to (\ref{pb}).  Applying (\ref{conn}), the contribution of the $i$-th point to this sum is\footnote{Strictly speaking, we would need to smooth the functions $W_C.$ To do this, smooth the connection $A$ by convolving with a delta sequence $\delta_\epsilon,$ and write $W_C^\epsilon(A) = W_C(A \star \delta_\epsilon)$ (where here $\star$ denotes convolution).  Applying $\Pi$ to the product $W_C^\epsilon W_{C^\prime}^\epsilon $ of two such functions and taking the limit $\epsilon \to 0,$  we obtain (\ref{pbi}).}

\begin{equation}\labell{pbi}\sum_{\alpha,\beta} (g^{-1})_{\alpha\beta}  \epsilon_i  \tr( \hol_{C,x_i}(A) e_\alpha) \tr( \hol_{C^\prime,x_i}(A) e_\beta)\end{equation}

\noindent where, as before, $e_\alpha$ is a basis for $\fg$ and $\epsilon_i = \pm 1$ is a sign arising from any symplectic form on  $\Sigma$ giving the orientation of $\Sigma,$ evaluated on the tangent vectors to the oriented curves; explicitly, if $v_i,v_i^\prime \in T\Sigma|_{x_i}$ are tangents to the oriented curves $C,C^\prime$ at $x_i,$ pointing in the direction of the orientations, then\footnote{Note that if the curves $C,C^\prime$ do not intersect, the corresponding Wilson loops Poisson commute; see \cite{weitsman} for a proof of this fact along the lines of the gauge theoretic computation given here.}

\begin{equation}\labell{epseq}\epsilon_i = {\rm sgn}( \omega_{x_i}(v_i,v_i^\prime))\end{equation}

\noindent where $\omega$ is a symplectic form on the oriented two manifold $\Sigma$ with $\int_\Sigma \omega > 0.$

If we choose the same representation for the Wilson loops as we did in defining the symplectic form, so that ${\rm tr} = {\rm Tr},$ we have

\begin{equation}\labell{gt} \sum_{\alpha,\beta} (g^{-1})_{\alpha\beta}  \tr(U e_\alpha) \tr (V e_\beta) = \tr (\pi(U) \pi(V))\end{equation} 

\noindent where $\pi : G \to \fg$ is the map given by composing the inclusion of $G$ into $GL(m,\C)$ given by the chosen representation, followed by the restriction to $GL(m,\C) \subset {\mathfrak gl}(m,\C)$ of the projection ${\mathfrak gl}(m,\C) \to \fg$ in the metric arising from   the trace ${\rm Tr}.$  

In the case where $G = GL(n,\C)$ and the representation is the defining representation, the map $\pi$ is the identity, and we obtain

\begin{equation}\labell{gtgln}\sum_{\alpha,\beta} (g^{-1})_{\alpha\beta}  \tr(U e_\alpha) \tr (V e_\beta)= \tr (UV).\end{equation}

Applying this to the Poisson bracket $\{ W_C , W_{C^\prime} \},$ equations (\ref{pb}) (\ref{pbi}) and (\ref{gtgln}) give

\begin{equation}\labell{pbgln}\{ W_C , W_{C^\prime} \} (A)= \sum_i \epsilon_i W_{C *_i C^{\prime}}(A)\end{equation}

\noindent where ${C *_i C^{\prime}}$ denotes the concatenation of the oriented curves $C$ and $C^{\prime}$ at the common basepoint $x_i.$  This agrees with Goldman's computation (Theorem 3.13 of \cite{goldman}).

In the case of $G = SU(2)$ and the defining representation, the map $\pi:SU(2) \to su(2)$ is given
by 

$$\pi(U) = \frac12(U - U^{-1}).$$

\noindent (cf. Corollary 1.10 of \cite{goldman}).  We therefore have

\begin{equation}\labell{pbsu2}\{ W_C  ,W_{C^\prime} \} (A)= \frac12 \sum_i \epsilon_i (W_{C *_i C^{\prime}}(A) - W_{C *_i \bar{C}^{\prime}}(A) )\end{equation}

\noindent  where we denote by $\bar{C}$ the curve $C$ with the orientation reversed.  This is in agreement with Theorem 3.17 of \cite{goldman}. 

Note that the Poisson bracket we have computed on $\cA(\Sigma)$ extends immediately (by the Leibniz rule) to functions given by polynomials in Wilson loops, so long as they obey (pairwise) the transverse intersection condition.

\begin{Remark} It is not so clear that the computation of the Poisson bracket on (nicely intersecting) Wilson loops suffices to determine its value on all smooth functions on $\cA(\Sigma),$ or on the moduli space of flat connections.  In finite dimensions, Nachbin's Theorem gives a criterion for an algebra of functions to be dense in the algebra of smooth functions on a smooth manifold:  

\begin{RefTheorem}[Nachbin \cite{nachbin}] 

Let $A$ be a subalgebra of the algebra $C^\infty(M)$ of smooth functions on a finite dimensional smooth manifold $M.$ Suppose that 

\begin{itemize}
\item For each $m \in M$ there exists $f \in A$ so that $f(m)\neq 0.$
\item For each pair $m, m^\prime \in M, $ with $m \neq m^\prime,$ there exists $f \in A$ so that $f(m) \neq f(m^\prime).$
\item For each $m \in M,$ and each $v \in TM|_m,$ with $v \neq 0,$ there exists $f \in A$ so that $df_m(v) \neq 0.$
\end{itemize}
Then $A$ is dense in the Frechet space $C^\infty(M).$ \end{RefTheorem}

In the case of smooth moduli spaces, this theorem can be applied.  But the algebra generated by Wilson loops given by the defining representation of $SU(2)$ does not separate tangent vectors at the equivalence class of the trivial representation,\footnote{This representation occurs only in a singular quotient of $\cA(\Sigma),$ so this argument is not quite on all fours, but is indicative of the issues to be dealt with.} due to the fact that elements of ${\mathfrak su}(2)$ are traceless in this representation. Similar problems occur for other groups. It is possible that including Wilson loops in all representations remedies this problem.\end{Remark}

\begin{Remark}[$SL(2,\R)$ and $SL(2,\C)$]  Note that as in the case of $G=SU(2)$, the Poisson bracket of Wilson loops in the fundamental representation of $SL(2,\R)$ or $SL(2,\C)$ is obtained from formula (\ref{gt}) with 

$$\pi(U) = U - \frac12 {\rm tr}(U) I.$$

Thus in all three cases we obtain

$${\rm tr}(\pi(U) \pi(V)) = {\rm tr}(UV) - \frac12 {\rm tr}(U) {\rm tr}(V)$$

and therefore

\begin{equation}\labell{su2alt} \{W_C, W_{C^\prime}\} = \sum_{i} \epsilon_i (W_C *_{i} W_{C^\prime} - \frac12 W_C W_{C^\prime}).\end{equation}

When $G = SU(2)$ (and also $G=SL(2,\R)$), this expression coincides with (\ref{pbsu2}), since for $V \in SU(2)$ or $V \in SL(2,\R),$ 

$$ V + V^{-1} = {\rm tr}(V) I,$$ so that

$${\rm tr}(UV) - \frac12 {\rm tr}(U) {\rm tr}(V) = {\rm tr}(UV) - \frac12 {\rm tr}(U (V+ V^{-1})) =\frac12 ({\rm tr}(UV) - {\rm tr}(UV^{-1})).$$

See \cite{goldman}.
\end{Remark}

\section{Chern Simons gauge theory in axial gauge on $\Sigma \times \R$}\labell{kqmsec}

Let $\Sigma$ be a compact, oriented two manifold, let $G$ be a Lie group, and consider the space of connections on the trivialized principal $G$-bundle on $\Sigma \times \R.$  Choose a representation of $G$; this gives a trace ${\rm Tr}$ on the lie algebra $\fg.$  The space of connections may be identified with $\Omega^1_c(\Sigma \times \R) \otimes \fg,$ and in these terms the Chern-Simons invariant of a connection $A \in \Omega^1_c(\Sigma \times \R) \otimes \fg$ is given by

$${\rm CS}(A) = \frac{1}{4\pi} {\rm Tr} \int_{\Sigma \times \R} A dA + \frac23 A^3.$$

As in \cite{kqm}, let $\bA\subset \Omega^1_c(\Sigma \times \R) \otimes \fg$ denote the subspace of the space of connections satisfying the axial gauge condition

$$A^t = 0$$

\noindent where $t$ is the coordinate on $\Sigma \times \R$ arising from projection to $\R.$

The Chern-Simons invariant of an element $A \in \bA$ is then\footnote{We absorb a factor of $2\pi$ into the coupling in the functional integrals.}

$$S(A) = \frac12{\rm Tr} \int_{\Sigma \times \R} A \wedge \frac{dA}{dt}\wedge dt.$$

As in \cite{kqm}, we consider the formal functional integral with action $S.$  Expectations in this formal functional integral can be computed using the propagator, the formal inverse of the operator appearing in the definition of $S,$
and given by (\ref{propagator}), viz

\begin{equation}\labell{propagator1} 
\Gamma^{ij}_{\alpha\beta} (x,t; x^\prime,t^\prime) = (g^{-1})_{\alpha\beta} \epsilon_{ij} \delta(x,x^\prime) u(t-t^\prime).
\end{equation}

\noindent Here we have chosen a basis $\{e_\alpha\}$ for $\fg,$ and define $g_{\alpha \beta} = {\rm Tr~}(e_\alpha e_\beta).$\footnote{The trace is taken in the representation used to define the Chern Simons function.}    Also $x,x^\prime \in \Sigma,$ $t,t^\prime \in \R,$ and\footnote{The addition of a constant to $u,$ although in principle allowable, results in no change to any expectation, due to the fact that $\epsilon_{ij}$ is skew.}

$$u(s) =  \left\{
	\begin{array}{ll}
		\frac12  & \mbox{if } s \geq 0 \\
		-\frac12 & \mbox{if } s< 0
	\end{array}
\right.$$

The operator $\Gamma$ gives well defined formulas, as in \cite{kqm}, for the expectations of polynomials in the gauge field $A,$ and these can be thought of as giving a linear functional on the algebra of such polynomials.\footnote{Again, as in Remark \ref{invtrmk}, we may write the operator $\Gamma$ in a form independent of choice of coordinates on $\Sigma,$ showing in particular that it depends only on an orientation on $\Sigma,$ not on the choice of local coordinates.}  We wish to find a way to define expectations of polynomials in Wilson loops, which are not of this form.  We proceed as in \cite{kqm}.

Let $C_1,\dots,C_n$ be closed, oriented, immersed curves in $\Sigma \times \R$, whose projections to $\Sigma$ intersect pairwise transversally at a finite number of points $x_1,\dots,x_n$ of $\Sigma,$ with inverse images $(x_1,t_1^\pm),\dots,(x_n,t_n^\pm) \in \Sigma \times \R.$  As in \cite{kqm}, the formal functional integral

$$   \int DA e^{- \lambda S(A)} W_{C_1} (A) \dots W_{C_n}(A)$$
for $\lambda \in \C$ should be taken to be

$$   \int DA e^{- \lambda S(A)} W_{C_1} (A) \dots W_{C_n}(A)= \langle W_{C_1}   \dots W_{C_n} \rangle_{A=0},$$

\noindent where the $\langle\cdot \rangle$ is a linear functional on the vector space generated by monomials in the $W_C$ involving curves satisfying our intersection conditions and given by
$$\langle W_{C_1}  \dots W_{C_n} \rangle (A) = \prod_i e^{\frac1{2\lambda} D_i} W_{C_1} (A) \dots W_{C_n}(A)$$

\noindent where

\begin{equation}\labell{op3} D_i=   \sum_{\alpha,\beta} \sum_{a,b} \sum_{s,s^\prime \in \{t_i^+,t_i^-\}}(g^{-1})_{\alpha\beta} \epsilon_{ab}  u(s - s^\prime)   \frac{\delta}{\delta A_\alpha^a(x_i,s)}\frac{\delta}{\delta A_\beta^b(x_i,s^\prime)}. \end{equation}

Note that the linear functional $\langle\cdot\rangle,$ which is defined on a space of functions on the space of connections, involves only functional differentiation, not functional integration.

Note also that the definition of $\langle\cdot\rangle$ does not contain terms referring to possible self-intersections of the curves $C_i,$ for reasons that will become clear when we compute the effect of the Poisson bivector field.  In the language of \cite{K2} and \cite{kqm}, we consider these self-intersections to be virtual intersections.

In \cite{kqm} we developed formulas to compute the value of a linear functional of the type $\langle\cdot\rangle$ in the case where $\Sigma \times \R$ was replaced by $\R^3.$  These same formulas carry over immediately to the case we are considering here, and can be summarized, using the standard diagrams in knot theory, as follows.

Consider first the case where $G = SU(2),$ and where we use the defining representation for both the Chern-Simons functional and the Wilson loops.  Then
we showed 

\begin{equation}\labell{urskeino}
D  (\overcrossing )= -1   (\overcrossing )+ 2 (\upupsmoothing) \end{equation}

\begin{equation}\labell{urskeinu}
D ( \undercrossing ) = 1  (\undercrossing )- 2 ( \upupsmoothing) .\end{equation}

We then showed how to compute the exponentials of $D,$ and proved

\begin{equation}\labell{skeinsu2over}
\langle \overcrossing \rangle = \Big(\cosh(\sqrt{3} \beta) - \frac{1}{\sqrt{3}}\sinh(\sqrt{3} \beta)\Big) (\doublepoint )
+ \Big(\frac{2}{\sqrt{3}}\sinh(\sqrt{3} \beta)\Big) (\upupsmoothing)\end{equation}

and

\begin{equation}\labell{skeinsu2under}
\langle \undercrossing \rangle =  \Big(\cosh(\sqrt{3} \beta)+\frac{1}{\sqrt{3}}\sinh(\sqrt{3} \beta)\Big) (\doublepoint )
- \Big(\frac{2}{\sqrt{3}}\sinh(\sqrt{3} \beta)\Big) (\upupsmoothing),\end{equation}

\noindent where $\beta = \frac1{2\lambda}.$

Here we have used the notation ($\doublepoint$) for a "bare" crossing; one where all the appropriate operators $D_i$ have already been applied.  This notation has been used by Kauffman \cite{K2} to denote a virtual crossing. Similarly ($\upupsmoothing$) for the "honest" smoothed intersection (in our case involving a "jump").

This formula applies also to the groups $SL(2,\C)$ and $SL(2,\R),$ again in the defining representation.

In the case where $G = SU(2)$ or $G=SL(2,\R)$, the Wilson loop is independent of orientation, since for $V \in SU(2)$ or $V \in SL(2,\R),$

$$ {\rm Tr}(V) = {\rm Tr}(V^{-1}).$$

We can then express formulas (\ref{skeinsu2over}) and (\ref{skeinsu2under}) (after a normalization $W_C \to -W_C$) in terms of unoriented curves as 

\begin{equation}\labell{skeinsu2over2}
\langle \overcrossing \rangle =  a (\smoothing) + b (\hsmoothing)
\end{equation}

\noindent and

\begin{equation}\labell{skeinsu2under2}
\langle \undercrossing \rangle = b (\smoothing) + a (\hsmoothing)
\end{equation}

\noindent where 

\begin{equation}\labell{a}
a = -\cosh(\sqrt{3} \beta) - \frac{1}{\sqrt{3}}\sinh(\sqrt{3} \beta)
\end{equation}

\noindent and

\begin{equation}\labell{b}
b = -\cosh(\sqrt{3} \beta) + \frac{1}{\sqrt{3}}\sinh(\sqrt{3} \beta).
\end{equation}

Likewise, if we take $G = GL(n,\C)$ or $G = U(n),$ we have the formulas

\begin{equation}\labell{gurskeino}
D ( \overcrossing)=  ( \upupsmoothing)\end{equation}

\begin{equation}\labell{gurskeinu}
D(\undercrossing) = - (\upupsmoothing).\end{equation}

We may again exponentiate\footnote{There are subtleties involved in computing the action of $D$ on the smoothings, which involve jumps, appearing in (\ref{gurskeino}) and (\ref{gurskeinu}).  See \cite{kqm} for the detailed computations.}  to obtain

 \begin{equation}\labell{oskeingln} 
 e^{-\beta \frac{n}2 }\langle \overcrossing \rangle = \Big( \cosh(\beta {\sqrt{\Delta_n}} ) -\frac{n}2  \frac{1}{\sqrt{\Delta_n}} \sinh(\beta{\sqrt{\Delta_n}} )\Big) (\doublepoint)  + \frac{2}{\sqrt{\Delta_n}} \sinh(\beta{\sqrt{\Delta_n}} ) (\upupsmoothing)
 \end{equation}
 
 \noindent while
 
  \begin{equation}\labell{uskeingln} 
 e^{\beta \frac{n}2 }\langle \undercrossing \rangle = \Big( \cosh(\beta {\sqrt{\Delta_n}} ) + \frac{n}2  \frac{1}{\sqrt{\Delta_n}} \sinh(\beta{\sqrt{\Delta_n}} )\Big) (\doublepoint)  - \frac{2}{\sqrt{\Delta_n}} \sinh(\beta{\sqrt{\Delta_n}} ) (\upupsmoothing),
 \end{equation}
 
 \noindent where $\Delta_n =  \frac{n^2}4 + 2.$
 
 Note that if we set $n=2,$ equations (\ref{oskeingln}) and (\ref{uskeingln}) differ from their counterparts (\ref{skeinsu2over}) and (\ref{skeinsu2under}) only by factors of $ e^{\pm\beta \frac{n}2 },$ corresponding to knot framing.
 
All these formulas are taken verbatim from \cite{kqm}.  In the next section, we will see how they give a star product on the space of connections on $\Sigma.$

\section{Chern Simons gauge theory in three dimensions and the star product on the space of connections in two dimensions}\labell{combisec} In this Section we show that the uncanny resemblance between the formulas (\ref{pbgln})  and ({\ref{pbsu2})   obtained for the Poisson bracket and the formulas (\ref{gurskeino}) - (\ref{gurskeinu}) and (\ref{urskeino}) - (\ref{urskeinu}) for the the action of the operator $D$ in both the $SU(2)$ and $GL(n,\C)$ cases are not coincidences; they arise because there is an identity between  Poisson brackets of Wilson loops in two dimensional gauge theory and expectations of related objects in Chern-Simons gauge theory.  This identity echoes the relation between the manifold invariants of Witten and Reshetikhin-Turaev and geometric quantization of moduli spaces.

To do this we compare the operators $D$ and $\Pi.$

Recall that as in (\ref{op}), the operator $D$ is given by

\begin{equation}\labell{op4} D=  \int dx   dt dt^\prime \sum_{\alpha,\beta} \sum_{a,b} (g^{-1})_{\alpha\beta} \epsilon_{ab}  u(t  - t^\prime) \  \frac{\delta}{\delta A_\alpha^a(x,t)}\frac{\delta}{\delta A_\beta^b(x ,t^\prime)}. \end{equation}

On the other hand, the bivector field $\Pi$ is given (\ref{pb0}) by the operator

\begin{equation}\labell{pb4}\Pi =   \int_\Sigma dx  \sum_{\alpha,\beta} \sum_{a,b}   (g^{-1})_{\alpha\beta} \epsilon_{ab} \frac{\delta}{\delta A_\alpha^a(x)}\frac{\delta}{\delta A_\beta^b(x)}.\end{equation}

Hence, if $C, C^\prime$ are two curves in $\Sigma$ intersecting transversally at a finite number of points in $\Sigma,$ and if we denote by $T_\epsilon$ the translation by $\epsilon$ in the $\R$ direction in $\Sigma \times \R,$ we have, for any $\epsilon > 0$ and any $A \in \cA(\Sigma),$

$$(m\Pi (W_C \otimes W_{C^\prime}))(A) =  \frac12 D (W_{T_\epsilon C} W_{T_{-\epsilon} C^\prime}) (\pi^* A),$$

\noindent where $\pi : \Sigma \times \R \to \R$ is the projection and $m$ denotes multiplication in the algebra of functions on $\cA(\Sigma).$  We summarize these results below.

\begin{Theorem}\labell{dispi} Let $f$ and $g$ be two polynomials in Wilson loops on $\Sigma$ corresponding to curves which intersect pairwise transversally.  Then
\begin{equation}\labell{dispieq} (m \Pi (f \otimes g))(A) =\frac12(D (T_\epsilon f  {T_{-\epsilon} g}) )(\pi^* A),\end{equation}
\end{Theorem}
\noindent Here the operator $T_\epsilon$ is defined on polynomials in the Wilson loops by defining

$$T_\epsilon (W_{C_1}\dots W_{C_n} ) =  (W_{T_\epsilon C_1}) \dots ( W_{T_\epsilon C_n} ) $$

\noindent for each monomial in the Wilson loops, and extending the definition to polynomials by linearity.\footnote{The factor of $\frac12$ in (\ref{dispieq}) is due to the difference between the action (by the product rule) of the symmetric second order differential operator $D$ on the product of the two functions $T_\epsilon f , T_{-\epsilon} g$ and the contraction of the skew bivector field $\Pi$ with $df \otimes dg.$}
 
It follows that

\begin{Theorem}\labell{edisepi} Let $f$ and $g$ be two polynomials in  Wilson loops on $\Sigma$ corresponding to curves which intersect pairwise transversally.  Then

$$ (m e^{h \Pi}  (f \otimes g))(A) =(e^{\frac{h}{2} D} (T_\epsilon f  {T_{-\epsilon} g}) )(\pi^* A),$$

\noindent where we consider all intersections of the loops appearing in $f$ to be virtual crossings, likewise all intersections of loops appearing in $g.$

\end{Theorem}

Now, on an affine space like $\cA(\Sigma),$ the formula of \cite{bffls} for the star product (in terms of Fedosov's construction, the star product associated with the symplectic connection invariant under affine translations; see \cite{f}, p. 222) is

\begin{Definition}\labell{stp}

The star product of two polynomials $f,g$ in Wilson loops that intersect pairwise transversally is given by (setting $\beta = \frac{h}{2}$)
\begin{equation}
( f \star_h g )(A) =  (m e^{h \Pi}  (f \otimes g))(A) = \langle T_\epsilon f T_{-\epsilon} g \rangle({\pi^*A}). 
\end{equation}
\end{Definition}

We verify directly that this product is associative where defined:

\begin{Theorem}\labell{assoc}  The star product is associative where defined; that is, if $u, v $ and $w$ are polynomials in Wilson loops in $\Sigma,$ corresponding to curves which intersect pairwise transversally, then

$$(u \star_h v) \star_h w = u \star_h (v \star_h w).$$\end{Theorem}

\begin{proof}  We have (recalling we set $\beta = \frac{h}2$)

$$ ((u \star_h v) \star_h w )(A) = \langle T_\epsilon (T_{\epsilon} u T_{- \epsilon}v) (T_{-\epsilon}w )\rangle (\pi^*A)= \langle (T_{2\epsilon} u) v (T_{-\epsilon}w)\rangle (\pi^*A)$$

while

$$( u \star_h (v \star_h w))(A) = \langle (T_\epsilon u) T_{-\epsilon} ( T_\epsilon v T_{-\epsilon} w)\rangle (\pi^*A)
= \langle (T_\epsilon u)  v (T_{-2\epsilon} w)\rangle (\pi^*A).$$

But, since $u(t-s)$ depends only on the sign of $t-s,$ 

$$\langle (T_{2\epsilon} u) v (T_{-\epsilon}w)\rangle (\pi^*A)=  \langle (T_\epsilon u)  v (T_{-2\epsilon} w)\rangle (\pi^*A).$$

\end{proof}

\subsection{Explicit formulas for the star product}\labell{ex1} Definition \ref{stp}, combined with formulas 
(\ref{skeinsu2over2}) and (\ref{skeinsu2under2}) (in the case of $SU(2), SL(2,\R),$ and $SL(2,\C)$) or formulas (\ref{oskeingln}) and (\ref{uskeingln}) (in the case of $GL(n,\C)$ or $U(n)$) give explicit formulas for the star product:

As a first example, consider the case where $G=SU(2).$  Let $C,C^\prime$ be two immersed curves in $\Sigma$ and suppose $C,C^\prime$ intersect transversally at one point $p, $  and that the tangents $v,v^\prime \in T\Sigma|_p$ to the oriented curves $C,C^\prime$ at $p$ satisfy $\omega_{p}(v,v^\prime) >0 $  for some symplectic form $\omega$ on $\Sigma$ with $\int_\Sigma \omega > 0.$  Then applying Theorem \ref{edisepi} and equation (\ref{skeinsu2over}), and recalling $\beta = \frac{h}{2},$ we have  

\begin{equation}\labell{ex1e} W_C \star_h W_{C^\prime} =  \Big(\cosh(\frac{\sqrt{3}}{2} h) - \frac{1}{\sqrt{3}}\sinh(\frac{\sqrt{3}}{2} h )\Big)W_C W_{C^\prime} + \Big(\frac{2}{\sqrt{3}}\sinh(\frac{\sqrt{3}}{2} h)\Big) W_{C *_p C^\prime}. \end{equation}

Note that for $h$ small we have, comparing with (\ref{su2alt}), 

$$W_C \star_h W_{C^\prime} = W_C W_{C^\prime} + h( W_{C *_p C^\prime} -\frac12 W_C W_{C^\prime} )+ O(h^2)= W_C W_{C^\prime} + h \{W_C ,W_{C^\prime} \} + O(h^2)$$

\noindent as expected.

More generally, the algorithm for computing the star product of any two monomials $f= W_{C_1} \dots W_{C_n} $ and $g = W_{C_1^\prime} \dots W_{C_m^\prime}$ can be given as follows.

Step 1.  Translate $C_1\dots C_n$ by $\epsilon$ in the vertical direction to obtain a collection of curves in $\Sigma \times \R.$  Likewise translate $C_1^\prime\dots C_m^\prime$ by $-\epsilon$ to obtain another collection of curves in $\Sigma \times \R.$  Consider the product 

$$P = W_{T_\epsilon C_1} \dots W_{T_\epsilon C_n }W_{T_{-\epsilon} C_1^\prime} \dots W_{T_{-\epsilon} C_m^\prime} .$$

Step 2.  For any pair $C_i, C_j^\prime$, and any intersection point $p \in C_i \cap C_j^\prime,$ remove $W_{T_\epsilon C_i}$ and $W_{T_{-\epsilon} C_j^\prime}$ from the product $P$ and replace it with 

$$c W_{{T_\epsilon C_i}}\circ_p W_{T_{-\epsilon} C_j^\prime} + d W_{T_\epsilon C_i} W_{T_{-\epsilon} C_j^\prime}.$$

\noindent where \begin{equation}\labell{concat}W_{{T_\epsilon C_i}}\circ_p W_{T_{-\epsilon} C_j^\prime} (\pi^* A) = {\rm tr} ({\rm hol}_{T_\epsilon C_i,p}(\pi^*A) 
 {\rm hol}_{T_{-\epsilon}C_j^\prime,p}(\pi^*A))\end{equation}

\noindent and where $c,d$ are given as in (\ref{skeinsu2over2}) and (\ref{skeinsu2under2}) (in the case of $SU(2), SL(2,\R),$ and $SL(2,\C)$) or formulas (\ref{oskeingln}) and (\ref{uskeingln})  (in the case of $GL(n,\C)$ or $U(n)$).

In other words, replace $P$ by

$$P_p  =  W_{T_\epsilon C_1} \dots \hat{W}_{{T_\epsilon C_i}} \dots W_{T_\epsilon C_n } W_{T_{-\epsilon} C_1^\prime} \dots  \hat{W}_{{T_\epsilon C_j^\prime}} \dots W_{T_{-\epsilon} C_m^\prime}(c W_{{T_\epsilon C_i}}\circ_p W_{T_{-\epsilon} C_j^\prime} + d W_{T_\epsilon C_i} W_{T_{-\epsilon} C_j^\prime}).$$

Step 3.  Repeat  Step 2 for all other intersection points.\footnote{The formula for the kind of Wilson loop arising from concatenation of two curves is not quite the same as that in (\ref{concat}) when one of the curves is itself a concatenation; it is the trace of a product of holonomies along segments.  We omit the precise formulas since the geometric picture is so clear.}

\section{Some Comments and Speculations:  Relation to Geometric Quantization; Quantization of Moduli spaces}\labell{gq}

\subsection{Geometric Quantization, the star product on Kahler manifolds, and Topological Quantum Field Theory}
 
The formula for the star product of two functions $f$ and $g$ on the space of connections arising from a functional integral in Chern-Simons gauge theory where translates of $f$ and $g$ are "stacked on top of each other" is not unexpected from the point of view of geometric quantization and topological quantum field theory.  

To see this, we first recall the construction of the star product for Kahler manifolds  \cite{bg,bms,guill,s}.  Let $(M,\omega)$ be a compact Kahler manifold, and let $L$ be a Hermitian holomorphic line bundle of curvature $\omega.$  The geometric quantization of $(M,k\omega)$ (where $k \in \Z$) is given in this Kahler polarization by the Hilbert space

$$ {\bf H} = H^0(M, L^k)$$

\noindent of holomorphic sections of $L^k;$ this space ${\bf H}$ is a finite dimensional subspace of the infinite dimensional Hilbert space $\Gamma(M,L^k)$ of $L^2$ sections of $L^k.$\footnote{The $L^2$ metric arises from the Hermitian metric on $L$ and the symplectic volume form on $M.$}  Let $p : \Gamma(M, L^k) \to {\bf H}$ denote the projection.

Given $f \in C^\infty(M)$, we obtain a Toeplitz operator $T_f$ on ${\bf H}$ by defining

$$T_f s = p (fs)$$

\noindent for any $s \in H^0(M, L^k).$  These Toeplitz operators generate a non-commutative, associative algebra, and we have  \cite{bg,bms,guill, s} for $f,g \in C^\infty(M),$ asymptotically as $k \to \infty,$  

\begin{equation}\labell{toep} T_f T_g \sim T_{f\star_{\frac1k} g}\end{equation}

\noindent where $f\star_h g$ is a star product on $(M,\omega).$

Now let $\Sigma$ be a compact Riemann surface and let $G$ be a compact Lie group; take for simplicity $G=U(n).$  We may consider the space $S = {\rm Hom}(\pi_1(\Sigma), G)/G$ as a space of stable bundles on the Riemann surface $\Sigma.$  This is a Kahler variety in its own right, and possesses a Hermitian line bundle $L$ of curvature given by the Kahler form.\footnote{Depending on the precise details, the variety $S$ may be singular; we do not expect the singularities to affect any of these considerations.}  It may therefore be quantized as above, and we may apply (\ref{toep}) to obtain a star product on the algebra of smooth functions on $S.$  Among these functions are the Wilson loops.

\begin{figure}\label{figure2}
       \includegraphics[width=.65\textwidth]{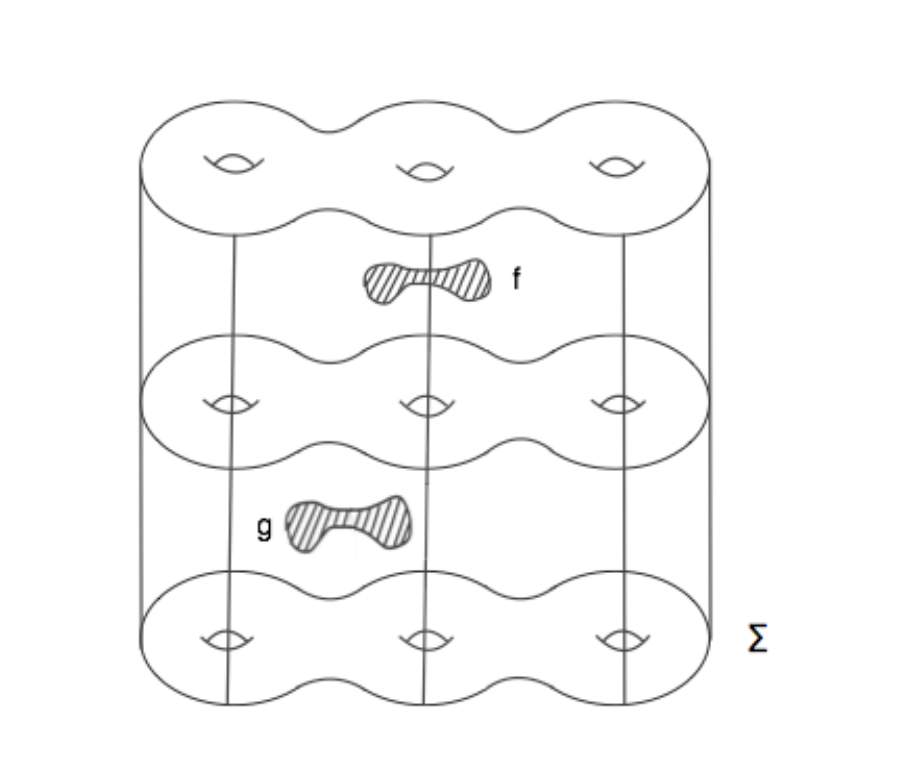}
\caption{The composition of $T_f$ and $T_g$ acting on the Hilbert space associated to $\Sigma.$}
   \end{figure}

On the other hand, the quantization ${\rm H}_S = H^0(S,L^k)$ is part of a three dimensional topological quantum field theory, giving the Hilbert space associated by the theory to the two manifold underlying the Riemann surface $\Sigma.$  One may reasonably hope that this topological quantum field theory would extend to three-manifolds with immersed Wilson loops.  Given a smooth function $f \in C^\infty(S),$ we would then expect to obtain an element of ${\rm H}_S \otimes {\rm H}_S$--that is, an operator $T_f : {\rm H}_S \to{\rm H}_S$--by placing the Wilson loops composing $f$ in a "block" or cobordism $\Sigma \times [0,1].$   One may also reasonably hope, from the axioms of Topological Quantum Field Theory,  that multiplication of two such operators $T_f, T_g$ would amount to composing the corresponding cobordisms; that is, to "stacking the blocks".\footnote{There is a vast literature on Skein Modules related to this type of construction in Topological Quantum Field Theory.}  (See Figure 1.)  We would expect these operators to coincide with the Toeplitz operators in the Kahler quantization ${\bf H}_S.$  The product obtained by "stacking the blocks" is then the operator product, and hence, by  (\ref{toep}),  gives the star product. Thus "stacking" the two functions $f$ and $g,$ to obtain the star product of $f$ and $g,$ as we did in Definition \ref{stp}, is precisely what would be predicted by Topological Quantum Field Theory, Geometric Quantization, and the relation between them in the case of moduli spaces.

In our Chern Simons functional integral in axial gauge, something like this rather subtle picture\footnote{Subject to issues of gauge invariance and the Second Reidemester move; see Section \ref{R2sec}.} is obtained from the simple relation between $\Gamma$ and $\Pi.$  We hope this serves as an illustration of the power of functional integrals when they can be treated mathematically.

\subsection{Gauge invariance, quantization of moduli spaces, and the Second Reidemeister move.}\labell{R2sec}

We saw in Section \ref{goldmansec} that our Poisson bracket on the Wilson loops on the space $\cA(\Sigma)$ of connections gave rise to Goldman's Poisson bracket on these same Wilson loops, considered as functions on the moduli spaces $\cA(\Sigma)//\cG$ of flat connections on $\Sigma;$ these moduli spaces are the character varieties $S = {\rm Hom}(\pi_1(\Sigma), G)/G$ considered above.  This follows from the general fact that on a Hamiltonian $G-$space $M,$ the Poisson bracket of invariant functions gives a Poisson bracket on the algebra of functions on the symplectic quotient $M//G,$ coinciding with the Poisson bracket coming from the symplectic form on the quotient.

Fedosov \cite{f2} shows that the deformation quantization of a Hamiltonian $G$-space also descends to the quotient in a similar way.   We may consider what this would mean for the case of the space of connections and its quotient by the gauge group.  Given two homotopic loops on $\Sigma,$ there is no necessary relation between the corresponding Wilson loops, considered as functions on $\cA(\Sigma).$  But as functions on the quotient $S,$ two homotopic loops give rise to identical functions.  So if our star product of such functions on $\cA(\Sigma)$ were to descend to a star product on the quotient, the star product of loops would have to be a homotopy invariant.  In particular, it would have to be invariant under the slide move in Figure 2---the two dimensional analog of the Second Reidemeister move for knot diagrams.
\begin{figure} 
       \includegraphics[width=.65\textwidth]{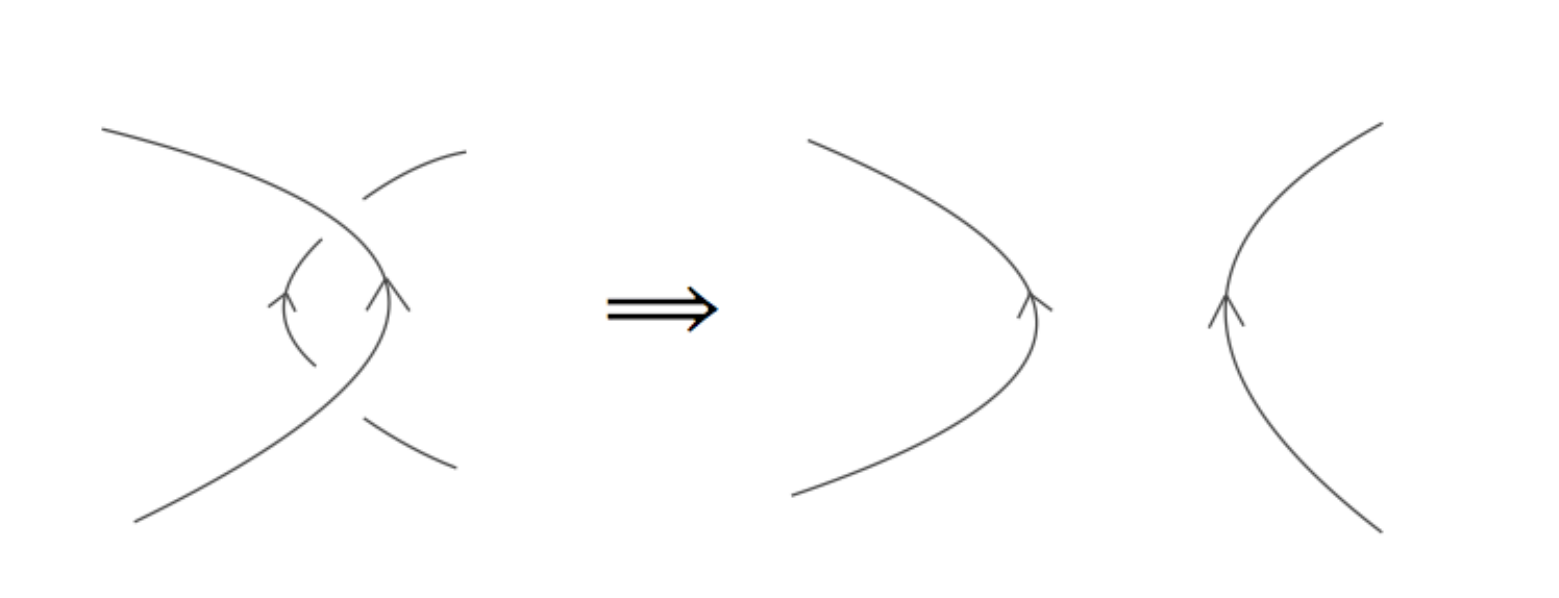}
\caption{The second Reidemeister move}
   \end{figure}
Now in \cite{kqm} we showed that in the case of $G=SU(2),$ the expectation $\langle\cdot\rangle$ was not invariant under the Second Reidemeister move; extracting a link invariant involved a shift in the value of the expectation of a single unknotted loop $\langle \bigcirc \rangle$ from the value expected in the Chern-Simons path integral, which is $\langle \bigcirc \rangle = 2.$\footnote{In \cite{kqm} we used a slightly different normalization for the expectation of Wilson loops associated with unoriented curves when $G = SU(2),$ which would change this expectation by a sign.} Since this value is also the value of the Wilson loop corresponding to a nullhomotopic simple closed curve, considered as a function on the space of flat connections, the same issue will affect the star product as well.  So our star product on $\cA(\Sigma)$ will not give a star product on $S;$
but the difficulty is due to a shift in something like a ground state energy.

From the point of view of \cite{f2}, there is a geometric explanation of why we do not obtain a star product on the quotient.  Fedosov's construction of the deformation quantization of a symplectic manifold $M,$ thought of as a star product on functions, involves a choice of a symplectic connection on $M.$  While the resulting algebra, thought of in terms of a sheaf of algebras on $M,$ is in an appropriate sense independent of this choice, this invariance is not so straightforward to describe in terms of the star product on functions.  

Now in our construction of the star product on $\cA(\Sigma),$ we made a particular choice of symplectic connection, arising from the affine structure of $\cA(\Sigma).$  This connection is not gauge invariant, and hence cannot be expected to yield a star product on $S.$  

Turning our attention back to Chern-Simons gauge theory, and motivated by the observation that gauge (non-)invariance may be behind the failure of the star product to be invariant under homotopy, it may be interesting to consider whether a more sophisticated treatment of the time independent gauge transformations in Chern Simons gauge theory in axial gauge may modify expectations in the theory
sufficiently so as to obtain invariance under the Second Reidemeister move, and give a geometric explanation for the shifted expectation $\langle \bigcirc \rangle = -(q^2 + q^{-2}) .$

\subsection{Skein modules and quantized Teichmuller space; possible relation to the volume conjecture}\labell{skein}

A relation similar to the one we have found between quantization of the space of connections and the Kauffman bracket is the relation between skein modules and the deformation quantization of Teichmuller spaces of \cite{cf,kashaev}.  A few references on this relation are \cite{tur1,bw,bfkb1,bfkb2,le,ps}.  Since Teichmuller space occurs as a component of the reduction of the space of connections by the gauge group, where $G=SL(2,\R),$ and since as remarked in \cite{kqm} and above, the expectations in Chern Simons gauge theory are the same for $SU(2), SL(2,\R),$ and $SL(2,\C),$ our methods should give a geometric proof of this relation using mathematical functional integrals, if we are able to find a gauge invariant symplectic connection, or at least a modification of the Chern Simons functional integral which is invariant under the second Reidemeister move, as explained in Section \ref{R2sec}.

We should note that results on the uniqueness of star products give reason to expect that the desired star product, as a star product on the space of connections, is equivalent to the star product in this paper.  Recall that given a star product $\star_h$ on a symplectic manifold $M$ and a differential operator $T$ on $C^\infty(M)[[h]]$ of the form

$$T = 1 + h \tau$$

\noindent where $\tau$ is another differential operator on $C^\infty(M)[[h]],$ a new star product $\star^\prime_h$ can be defined by

$$ f \star^\prime_h g = T^{-1} (Tf \star_h Tg)$$

\noindent for $f, g \in C^\infty(M)[[h]];$ two such star products are called {\em equivalent}.  Then \cite{bcg,deligne, gr} if $H^2(M) = 0,$ any two star products are related by a combination of equivalence and a change of the parameter $h.$\footnote{See \cite{gr}, Corollary 9.5.} Since $\cA(\Sigma)$ is contractible (though infinite dimensional), we may hope for a similar result here.

In a yet more speculative vein, the relation between the Kauffman bracket and quantized Teichmuller space, and more generally the coincidence of the expectations of loops in Chern Simons gauge theory for the group $SU(2)$ with those for Chern-Simons gauge theory for the groups $SL(2,\R)$ and $SL(2,\C),$ may be related to the volume conjecture \cite{kash,mura,muramura}

 \appendix
 \section{Some comments and speculations on a coproduct structure}
 
 In \cite{tu} Turaev showed that the Goldman bracket on the algebra of oriented loops on a two manifold could be supplemented with a cobracket, forming a bialgebra structure.
 
Something like this may be visible in our gauge theoretic construction.  In particular, if we consider the case $G = U(n)$ and take $n\to \infty$, the monomials in the Wilson loops might be expected to be linearly independent, and form a basis for a vector space $W$ of functions on $\cA(\Sigma).$  If we take the duals of those functions to be a basis for the dual vector space $W^\star,$ then we obtain a formula similar to (\ref{pbgln}), giving a bracket $W^\star\otimes W^\star \to W^\star.$  This bracket is equivalent to a cobracket on $W,$ which is the Turaev cobracket.

The bracket on $W^\star$ can be exponentiated to give an associative product on $W^\star,$ in line with the formulas in Definition \ref{stp} and given more explicitly in Section \ref{ex1}.  This in turn gives a coproduct on $W.$

\begin{Question}  Do the product and coproduct combine to make the loops on a compact oriented two manifold into a Hopf Algebra??

\end{Question}

If we could overcome the difficulties of gauge invariance, we might expect to obtain also a product and coproduct on the homotopy classes of loops on a compact oriented Riemann surface; equivalently, on the Wilson loops, considered as functions on $\cA(\Sigma)//\cG.$  In this case, there is a unit (consisting of the trivial loop) and counit (consisting of its dual), and likewise an inverse (by reversing the orientation of the loop) and antipode.  We may ask whether these fit into any kind of Hopf algebra structure.

\begin{Question} Is there an analogous construction for unoriented loops?\end{Question}

Of course all of these speculations ignore transversality issues and other troubles, which mean that none of our products or coproducts are really well-defined.  They are very far from established facts.

\subsection{A geometric coproduct}  There is also another coproduct construction on the algebra of loops, arising from Bott periodicity.  Again, let us take $G = U(n).$

The space $\cA(\Sigma)$ is a contractible space on which the group $\cG_\bullet$ of based maps from $\Sigma$ to $G$ acts freely.  The functions $W_C$ are $\cG$-invariant functions on this space.  We can therefore consider them as functions on the classifying space $B \cG =  B({\rm Map}(\Sigma, U(n)))= {\rm Map}(\Sigma, BU(n) ).$  Note that if we are given a curve $C$ in $\Sigma$ and a point of ${\rm Map}(\Sigma, BU(n) ), $ we obtain by composition an element of $\Omega BU(n) = U(n);$ taking the trace we get a complex number, and the resulting function on ${\rm Map}(\Sigma, BU(n) )$ (or, equivalently,  $\cG$-invariant function on $\cA(\Sigma),$ is the Wilson loop $W_C.$  If we take $n \to \infty,$ these are functions on $ B({\rm Map}(\Sigma, U )) = {\rm Map}(\Sigma, BU ) = {\rm Map}(\Sigma, \Omega_0 U ),$ where 
$ \Omega_0 U $ is the space of based loops on $U$ which are homotopically trivial, by Bott periodicity.  The space $ \Omega_0 U $ is a topological group (from the group structure on $U$) as well as an $H$-space (from concatenation of loops).  This means that the space $ B({\rm Map}(\Sigma, U ))={\rm Map}(\Sigma, \Omega_0 U )$ also has two $H-$space structures.  So the space of functions on ${\rm Map}(\Sigma, \Omega_0 U )$ acquires two coproduct structures, one of which is coassociative and one of which is coassociative "up to homotopy".
A third construction of a coproduct arises from  a map  $BU \times BU \to BU$  we can write directly:  Identifying $BU(n) = Gr(n,\infty)$ with the Grassmannian, consider the map $BU(n) \times BU(n) \to BU(2n)$ arising from the map taking two $n$-dimensional subspaces of Hilbert space $x, y \in Gr(n,\infty)$ to $ x \oplus y \in Gr(2n,\infty).$  Taking the limit $n \to \infty,$ we obtain a map $BU \times BU \to BU.$  And this again gives a type of product structure on ${\rm Map}(\Sigma, BU),$ again giving a coproduct on the functions on ${\rm Map}(\Sigma, BU).$\footnote{This construction on $BU$ is well known (see e.g. \cite{jfa}), but its application to ${\rm Map}(\Sigma,BU)$ is not, to my knowledge.}  Note that all these coproducts are coassociative, but not co-commutative; and they do not come in a parametrized family in any obvious way.\footnote{It is possible to introduce an integer parameter into the third construction by modifying the map to $(x,y) \to x^{\oplus k} \oplus y.$}

 \begin{Question} Do any of these geometric coproducts have anything to do with the Turaev cobracket?
\end{Question}

 \end{document}